\documentclass[11pt,a4paper]{article}
\usepackage{mathrsfs}
\usepackage{amsfonts}
\usepackage[leqno]{amsmath}
\usepackage{graphicx}
\usepackage{latexsym}
\usepackage{amsmath,amsfonts,amssymb,amsthm,mathrsfs,euscript,makeidx,color}
\usepackage{enumerate}

\usepackage{indentfirst}
\usepackage{epstopdf}

\usepackage[colorlinks,linkcolor=blue,anchorcolor=green,citecolor=red]{hyperref}

\def\sqr#1#2{{\vcenter{\vbox{\hrule height.#2pt
				\hbox{\vrule width.#2pt height#1pt \kern#1pt \vrule width.#2pt}
				\hrule height.#2pt}}}}
\def\5n{\negthinspace \negthinspace \negthinspace \negthinspace \negthinspace }
\def\4n{\negthinspace \negthinspace \negthinspace \negthinspace }
\def\3n{\negthinspace \negthinspace \negthinspace }
\def\2n{\negthinspace \negthinspace }
\def\1n{\negthinspace }

\def\dbR{\mathbb{R}}

\def\ds{\displaystyle}

\def\ns{\noalign{\ss}}

\def\ss{\smallskip}

\def\q{\quad}
\def\qq{\qquad}

\def\({\Big (}
\def\){\Big )}
\def\[{\Big[}
\def\]{\Big]}

\def\d{\delta}

\def\bde{\begin{definition}\label}
	\def\ede{\end{definition}}
\def\be{\begin{equation}}
\def\bel{\begin{equation}\label}
\def\ee{\end{equation}}
\def\bt{\begin{theorem}\label}
	\def\et{\end{theorem}}
\def\bc{\begin{corollary}\label}
	\def\ec{\end{corollary}}
\def\bl{\begin{lemma}\label}
	\def\el{\end{lemma}}
\def\bp{\begin{proposition}\label}
	\def\ep{\end{proposition}}
\def\bas{\begin{assumption}\label}
	\def\eas{\end{assumption}}
\def\br{\begin{remark}\label}
	\def\er{\end{remark}}
\def\bex{\begin{example}\label}
	\def\ex{\end{example}}
\def\ba{\begin{array}}
	\def\ea{\end{array}}
\def\ben{\begin{enumerate}}
	\def\een{\end{enumerate}}

\def\square#1{\vbox{\hrule\hbox{\vrule height#1%
			\kern#1\vrule}\hrule}}
\def\rectangle#1#2{\vbox{\hrule\hbox{\vrule height#1%
			\kern#2\vrule}\hrule}}

\font\tenbb=msbm10 \font\sevenbb=msbm7 \font\fivebb=msbm5

\newfam\bbfam
\scriptscriptfont\bbfam=\fivebb \textfont\bbfam=\tenbb
\scriptfont\bbfam=\sevenbb

\newtheorem{theorem}{\indent Theorem}[section]
\newtheorem{definition}[theorem]{\indent Definition}
\newtheorem{proposition}[theorem]{\indent Proposition}
\newtheorem{corollary}[theorem]{\indent Corollary}
\newtheorem{lemma}[theorem]{\indent Lemma}
\newtheorem{remark}[theorem]{\indent Remark}
\newtheorem{example}[theorem]{\indent Example}

\newtheorem{assumption}[theorem]{\indent Assumption}

\makeatletter

\@addtoreset{equation}{section}
\makeatother

\def\bea{\begin{equation*}}
\def\eea{\end{equation*}}

\def\bel{\begin{equation}\label}
\def\eel{\end{equation}}

\def\ba{\begin{array}}
\def\ea{\end{array}}
\newcommand{\ad}{&\!\!\!\displaystyle}

\def\({\Big (}
\def\){\Big )}
\def\[{\Big[}
\def\]{\Big]}
\def\q{\quad}
\def\qq{\qquad}

\def\d{\delta}

\def\ds{\displaystyle}

\def\ns{\noalign{\smallskip}}

\begin{document}
	
\title{Uniqueness of Dissipative Solution for Camassa-Holm Equation with Peakon-Antipeakon Initial
Data }

\author{Hong Cai \footnote{Department of Mathematics and Research Institute for Mathematics and Interdisciplinary Sciences, Qingdao University of Science and Technology, Qingdao, Shandong, P.R. China, 266061. Email: caihong19890418@163.com. Partially supported the National Natural Science Foundation of China (No. 11801295) and the Shandong Provincial Natural Science Foundation, China (No. ZR2018BA008).},\quad
Geng Chen \footnote{Department of Mathematics, University of Kansas, Lawrence, KS 66045, U.S.A. Email: gengchen@ku.edu. Partially supported by NSF grants DMS-1715012 and DMS-2008504.} ,\quad
Hongwei Mei \footnote{Department of Statistics, Rice University, Houston, TX 77705, U.S.A. Email: hongwei.mei@rice.edu.}}

\maketitle
\begin{abstract}
We give a proof for the uniqueness of dissipative solution for the Camassa-Holm equation
with some peakon-antipeakon initial data following Dafermos' earlier result in \cite{Daf} on the Hunter-Saxton equation.
Our result shows that two existing global existence frameworks, through the vanishing viscosity method by Xin-Zhang in \cite{XZ1} and the transformation of coordinate method for dissipative solutions by Bressan-Constantin in \cite{BC3}, give the same solution, for a special but typical initial data forming finite time gradient blowup. \bigbreak
\noindent


\end{abstract}
\section{Introduction}

We consider the Cauchy problem for Camassa--Holm equation
  \begin{equation}\label{CH-eq}
 u_t+\(\frac{u^2}2\)_x=-P_x,\qquad u(0,x)=u_0(x)\in H^1(\dbR),
 \end{equation}
where the nonlocal source term $P$ is defined as a convolution
\begin{equation}\label{CH-eq2}
P:=\frac12e^{-|x|}*\(u^2+\frac{u_x^2}2\).
\end{equation}
The Camassa--Holm equation was originally derived to model the surface waves  where
the term $u=u(t,x)$ represents the horizontal velocity of the fluid (see \cite{CH}).\ss

It is well known that the solution of Camassa-Holm equation might form finite time singularities. Or more precisely, the gradient of solution $u$ might blow up in finite time which leads to the wave breaking phenomena. Therefore it is imperative to consider the weak solutions in the space of $H^1(\dbR)$ at any time $t$. However, weak solutions of the initial value problem \eqref{CH-eq} naturally lose uniqueness after singularity formation. This nonuniqueness phenomena can be intuitively explained in the following way. Since $H^1(\dbR)\hookrightarrow C^{\frac12}(\dbR)$ by the Sobolev embedding theorem, roughly speaking, the characteristic equation of \eqref{CH-eq} behaves similarly to the ODE
$$\frac{dx}{dt}=\sqrt{|x|},$$
which has infinite many solutions when $x(0)=0$. They are corresponding
to multiple solutions (such as conservative and dissipative solutions) of Camassa-Holm equation after singularity formation.


 The global existence of ($H^1$) weak solutions for the Camassa--Holm equation  has been well established by two  frameworks: one applies the method of vanishing viscosity (see \cite{XZ1}) and the other one is to introduce a new semilinear system on new characteristic coordinates (see \cite{BC2, BC3}).
Especially, in \cite{BC2, BC3}, by the transformation of coordinates method, the existences of energy conservative and (fully) dissipative solutions have both been established. Furthermore, the existence of $\alpha$-dissipative solution, which has partial energy dissipation, has also been established in \cite{GHR}. The dissipation rate $\alpha$ might vary from $0\%$ to $100\%$, i.e. from conservative to fully dissipative solutions. Here we use the terminology fully dissipative solution to distinguish from the $\alpha$-dissipative solution with $\alpha<100\%$. Without confusion, throughout this paper, the dissipative solution always means the fully dissipative one. One can find other global $H^1$ existence results, such as in \cite{HR}. Some other interesting results can be found in \cite{HHXR2008,HHXR2009,KG2016}.

To select a physical solution for the Camassa-Holm equation, one needs to assume an additional physical admissible condition such as the conservation or full dissipation of energy. This additional condition will help selecting a unique solution flow after formation of cusp singularity.

The uniqueness for the energy conservative solution was first proved  in \cite{BCZ}, where Dafermos' framework in \cite{Daf} for the dissipative solution of Hunter-Saxton equation was used as one of the main techniques.

In this paper, we focus on the uniqueness of energy dissipative solution for the Camassa-Holm equation. We will present a proof on the uniqueness of dissipative solution with peakon-antipeakon initial data. For this speical initial data, we show that there is an easy way to apply Dafermos' framework in \cite{Daf}. The Cauchy problem considered in this paper is a typical problem forming cusp singularity. So our result indicates that solutions obtained from the vanishing viscosity method by Xin-Zhang in \cite{XZ1} and the transformation of coordinates method by Bressan-Constantin in \cite{BC3} are the same.

\begin{figure}[htp] \centering
		\includegraphics[width=.3\textwidth]{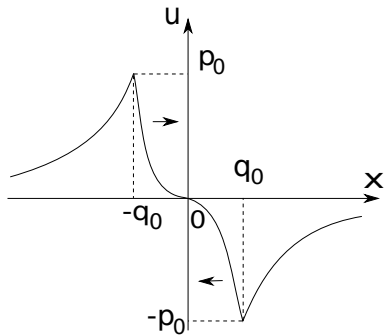}
		\caption{The Initial data. Two interacting peakons.\label{fig1}}
	\end{figure}

Now let's introduce the peakon-antipeakon initial data, as following
\bel{initialu}u_0(x)=\frac12p_0\(e^{-|x+q_0|}-e^{-|x-q_0|}\)\eel
for some positive constants $p_0, q_0>0$.
These initial data describe two interacting peakons (one bright and one dark). See Figure \ref{fig1}.  This special initial value problem enjoys a unique regular Lipschitz continuous solution before two peakons interact with each other. In fact, the cusp singularity forms at this interaction. Furthermore, the blowup time can be also explicitly calculated. We address the uniqueness of dissipative solution to
\eqref{CH-eq}, \eqref{CH-eq2} and \eqref{initialu}.

Now we first define the dissipative solution for Camassa-Holm equation, then introduce the main uniqueness theorem.
\begin{definition} The function $u=u(t,x)$, defined on $[0,\infty)\times\dbR$, is called a (global) solution of  \eqref{CH-eq}--\eqref{CH-eq2} if it is H\"{o}lder continuous  such that  $u(t,\cdot)\in H^1(\dbR)$ at every $t$, and the map $t\mapsto u(t,\cdot)$ is Lipschitz continuous  from $[0,\infty)$ to $L^2(\dbR)$ and satisfies the initial condition  $u(0,x)=u_0(x)$ together with the following equality between functions in $L^2(\dbR)$:
	\bel{weak}
	\frac{d}{dt}u=-uu_x-P_x,\quad  \hbox{ for }\ a.e.\  t\geq 0.
	\eel
	
Moreover, the solution $u(t,x)$ is said to be dissipative if it satisfies the Oleinik type inequality
\bel{ol}u_x(t,x)\leq C(1+t^{-1}),\quad t>0,\quad x\in\mathbb{R},\eel
for some constant $C$, and the total energy
	\bel{edef}
	E(t):=\int_{\dbR}\Big(u^2(t,x)+u_x^2(t,x)\Big)dx
	\eel
satisfies
$$E(t)\leq E(0)=:E_0, \quad\hbox{for any }\ t\geq0.$$
\end{definition}

\begin{theorem} \label{mainthm}
The dissipative solution $u=u(t,x)$ of \eqref{CH-eq}--\eqref{CH-eq2} with the peakon-antipeakon initial data \eqref{initialu} is unique.
\end{theorem}

%

\begin{remark}
The authors noticed a preprint posted by Jamr\'oz (see \cite{Jam} in its fifth version now), who claims the uniqueness of energy dissipative solution for the Camassa-Holm equation with general initial data. In contrast, we give a different and simpler proof for the Cauchy problem with peakon-antipeakon initial data.

In the definition of dissipation solution, \eqref{ol} can be changed to  that $u_x$ is bounded above in any bounded time interval as in \cite{Daf}, while the uniqueness theorem still holds. The current definition includes both solutions in \cite{XZ1} and \cite{BC3}.
\end{remark}

We divide the rest of the paper into $4$ sections. In Section 2, we review the regular solution before the peakon interaction. In Section 3, we introduce a transformation showing why Dafermos' idea on Hunter-Saxton equation can be applied to  \eqref{CH-eq}--\eqref{initialu}. In Section 4, we prove the existence and uniqueness of characteristic. The proof of Theorem \ref{mainthm} will be given in Section 5.

\section{Solution before blowup}\label{sec_before}

The existence of a unique solution for  \eqref{CH-eq}--\eqref{initialu} before blowup has been well established. In fact, the solution can be written out explicitly. Here we review this result using the reference \cite{BC3}.

As shown in \cite{BC3}, there exists an explosion time $T_0$ and a unique solution for \eqref{CH-eq}--\eqref{initialu} before blowup time $T_0$ satisfying
\bel{solu}u(t,x)=\frac12p(t)\(e^{-|x+q(t)|}-e^{-|x-q(t)|}\)\eel
with $p,q$ Lipschitz continuous such that $p(0)>0, q(0)>0$, and
\bel{initialvalue}\left\{\ba{ll}
\ad T_0=\frac1{H_0}\ln\(\frac{p_0+H_0}{p_0-H_0}\),\\
\ns\ad p(t)= H_0\frac{(p_0+H_0)+(p_0-H_0)e^{H_0t}}{(p_0+H_0)-(p_0-H_0)e^{H_0t}},\quad t\in [0,T_0),\\
\ns\ad q(t)=q_0+\ln\(\frac{(p_0+H_0)e^{-H_0t/2}+(p_0-H_0)e^{H_0t/2}}{2p_0}\),\quad t\in [0,T_0),\ea\right.\eel
where
$H^2_0=p^2_0(1-e^{-2q_0})$. And $(p(t),q(t))$ is the solution of the following ordinary differential equations
\bel{ODEpq}\left\{\ba{ll}\ad\dot q=\frac p2(e^{-2q}-1),\\
\ns\ad \dot p=\frac{p^2}2e^{-2q},\ea\right.\q\text{with}\q p(0)=p_0,\ q(0)=q_0.\eel
Note that
$p(t), q(t)> 0$ for $t\in[0,T_0)$.
Moreover, the energy is conserved before $T_0$, i.e.
$$p^2(t)(1-e^{-2q(t)})=H^2_0,\q0\leq t<T_0.$$
Simple computation from \eqref{initialvalue} yields that
\bel
{pqlim}\lim_{t\rightarrow T_0^-}p(t)=\infty,\q\lim_{t\rightarrow T_0^-}q(t)=0\quad \hbox{and}\quad\lim_{t\rightarrow T_0^-}p(t)\, q(t)=0.\eel

Throughout the following, we use $x_\xi(\cdot)$ to denote the characteristic initiated from $(0,\xi)$, i.e.
$$\dot x_\xi(t)=u(t,x_\xi(t)),\q x_\xi(0)=\xi.$$
By \eqref{ODEpq}, it is easy to calculate that $x=-q(t)$ and $x=q(t)$ are both characteristics, i.e.
$x_{-q_0}(t)=-q(t)$ and $x_{q_0}(t)=q(t)$. Thus, in view of \eqref{pqlim}, we have
\begin{equation}\label{charq0}
\lim_{t\rightarrow T_0^-}x_\xi(t)=0\q \hbox{for all}\q\xi\in[-q_0,q_0].
\end{equation}
Furthermore, using \eqref{solu} and \eqref{pqlim}, we can see that, as $t\rightarrow T_0^-$,
\bel{uxt0}
\left\{
\begin{array}{ll}
u_x (t, x_\xi(t))\rightarrow 0,& \hbox{when}\ \xi<-q_0\ \hbox{or}\ \xi>q_0,\\[2mm]
|u_x (t, x_\xi(t))|\rightarrow \infty,& \hbox{when}\ -q_0<\xi<q_0,
\end{array}
\right.
\eel
and
\bel{ut0}
u(t, x_\xi(t))\rightarrow 0,\qquad \hbox{for any}\ \xi\in(-\infty,\infty).
\eel
This means that all the characteristics starting from $[-q_0,q_0]$ will merge into one point 0 at the explosion time $T_0$. The non-uniqueness phenomenon happens after $T_0$.

In this paper, we will prove the uniqueness of dissipative solution for \eqref{CH-eq}--\eqref{initialu}, where the solution satisfies $$u(t,x)=0,\quad \hbox{when}\quad t\geq T_0.$$ So solutions established in \cite{XZ1} and \cite{BC3} both equal to this unique solution.

\begin{remark}{\rm
In \cite{XZ2}, the authors proved a weak-strong uniqueness result for the solution from the varnishing viscosity method.
This shows the uniqueness of dissipative solution before $T_0$.
Unfortunately, this result does not apply to the problem \eqref{CH-eq}--\eqref{initialu} considered in this paper. In fact, by  \eqref{solu} and \eqref{pqlim}, $\|u_x(\cdot,t)\|_{L^\infty}\simeq p(t) \simeq1/(T_0-t),$ as $t\rightarrow T_0^-$  which is not $L^2_{loc}([0,T_0])$.  So Theorem 2 in \cite{XZ2} fails to apply to this case.}
\end{remark}

\section{Why Dafermos' idea in \cite{Daf} applies to  \eqref{CH-eq}--\eqref{initialu}?\label{section2}}
In this section, let's introduce a transformation in order to showing why Dafermos' idea on the Hunter-Saxton equation can be applied to the Camassa-Holm equation with initial data \eqref{initialu}.

Let's always denote
$$ w(t,x):= u_x(t,x).$$
Then for
the Hunter-Saxton equation, one has
\begin{equation}\label{hsw}
 w_t+u w_x=-\frac12 w^2.
\end{equation}
Taking advantage of the clean representation, one can  explicitly solve $w(t)$ along the characteristics and identify its explosion time. The proof in \cite{Daf} relies on this explicit information of solution along characteristics.

However, for the Camassa-Holm equation, the equation of $w=u_x$ changes to
\begin{equation}\label{hsw1}
w_t+uw_x=-w^2-P_{xx}=-\frac 12{w^2}+u^2-P.
\end{equation}
We notice that the  nonlocal term $P$  involving $H^1$-norm  of $u$ might not be small, even for the solution of  peakon-antipeakon case  when $t\rightarrow T_0-$, i.e. before the formation of cusp singularity. So \eqref{hsw1} is quite different from \eqref{hsw}.

While if we view \eqref{hsw1} in a different way, we can find some similarity between the Camassa-Holm and Hunter-Saxton equations. Now we introduce a transformation
$$
\alpha(t,x):=w(t,x)-\int_{-\infty}^x u(t,y)dy.$$
For the solution of \eqref{CH-eq}--\eqref{initialu} when $t\in[T_0-\varepsilon,T_0+\varepsilon]$ with small $\varepsilon$,
we will show that $\alpha$ behaves similarly as $w$ in the Hunter-Saxton equation \textcolor{red}.

From \eqref{CH-eq} and the definition of $P$, we know that
\bel{info}\left\{\ba{ll} \ad u_{tx}+uu_{xx}+u_x^2=-P_{xx},\\
\ns\ad P-P_{xx}=u^2+\frac{u_x^2}2,\\
\ns\ad \int_{-\infty}^xu_t(t,y)dy=-P-\frac{u^2}2.\ea\right.\eel
 Write $\alpha_\eta(t):=\alpha(t,x_\eta(t))$ where $x_\eta(\cdot)$ is the characteristic starting at $(0,\eta)$, i.e.
 $$\dot x_\eta(t)=u(t,x_\eta(t))=:u_\eta(t),\q x_\eta(0)=\eta.$$
By \eqref{info}, we assume the solution is regular enough, then simple calculations yield that,
\begin{equation}
\label{hsw2}
\ba{ll}\ds\dot\alpha_\eta(t)\ad=u_{tx}(t,x_\eta(t))+u_{xx}(t,x_\eta(t))\dot x_\eta(t)-\int_{-\infty}^{x_\eta(t)}u_t(t,y)dy-u(t,x_\eta(t))\dot x_\eta(t)\\
\ns\ad=-P_{xx}(t,x_\eta(t))-u_x^2(t,x_\eta(t))-u_{xx}(t,x_\eta(t))u_\eta(t)+u_{xx}(t,x_\eta(t))u_\eta(t)\\
\ns\ad\q+P(t,x_\eta(t))+\frac12u_\eta^2(t)-u_\eta^2(t)\\
\ns\ad=-\frac{u_x^2(t,x_\eta(t))}2+\frac12u_\eta^2(t)\\
\ns\ad=-\frac12\(\alpha_\eta(t)+\int_{-\infty}^{x_\eta(t)}u(t,y)dy\)^2+\frac12u_\eta^2(t).\ea
\end{equation}

To prove uniqueness for \eqref{CH-eq}-\eqref{initialu}, one only has to consider the solution in $t\in[T_0-\delta,T_0+\delta]$ for a fixed small positive constant $\delta$.
In this time interval, it is expected that $\Vert u\Vert_{L^1\cap L^\infty}\lesssim p(t)q(t)$ with  $p(t)q(t)$ small. Hence the differences between equations \eqref{hsw} and \eqref{hsw2} are small. In fact,  in \eqref{hsw2}, the terms $\int_{-\infty}^{x_\eta(t)}u(t,y)dy$ and  $\frac12u_\eta^2(t)$ are both small near time $T_0$. 

Now, although it is hard to
solve $\alpha$ explicitly, a good estimate  by \eqref{hsw2} similar as by \eqref{hsw} for  the Hunter-Saxton equation can be found when we start from the regular solution at $t=T_0-\delta$. So we can apply Dafermos' framework to show that $u\equiv 0$ is the only solution in a short time interval after $t=T_0$ by some subtle estimates. Then the solution will keep zero for all time after $T_0$.

The idea given in this section could give a proof of the main theorem, but we find it may be more convenient to prove our result directly using $w$ instead of $\alpha$. So we will still use $w$ for the rest of the paper.

\section{Characteristics}
The existence of characteristics for \eqref{CH-eq}--\eqref{initialu} can be easily proved by the H\"older continuity of solution. In fact, we can directly use the Lemma 3.1 in \cite{Daf}.


\begin{lemma}[\cite{Daf}]\label{chara}
Assume that the continuous function $u(t,x)$ is a (weak) solution of the balance law
$$
\partial_t u(t,x)+\partial_x(\frac{1}{2}u^2(t,x))=g(t,x)
$$
on $[0,\infty)\times(-\infty,\infty),$ where $g(t,x)$ is a bounded measurable function such that
$g(t,\cdot)$ is continuous on $(-\infty,\infty)$ for fixed $t\in[0,\infty).$
Let $x=z(t)$ be a characteristic, namely a trajectory of the ordinary differential equation
$\dot x=u(t,x)$ on some time interval $[a,b]$. Then $z(t)$ together with the function
$v(t)=u(t,z(t))$ satisfy the system of ordinary differential equations
\begin{equation}\left\{
\begin{array}{rcl}
\dot{z}&=&v,\\
\dot{v}&=&g(t,z)
\end{array}
\right.
\end{equation}
on $[a,b]$. In particular, $v$ is Lipschitz continuous and $z$ is continuously
differentiable, with Lipschitz continuous derivative on $[a,b]$.
\end{lemma}

For Camassa-Holm equation,
$g(x,t)=-P_x$ is continuous, so this lemma provides the existence of characteristics for  \eqref{CH-eq}--\eqref{initialu}. Thus, by Lemma \ref{chara}, any characteristic $x=x_\zeta(t)$ starting at $(0,\zeta)$ is a continuously differentiable function, with \begin{equation}\label{char1}
\dot x_\zeta(t)=u_\zeta(t),\end{equation}
where
	$u_\zeta(t)=u(t,x_\zeta(t))$ is a Lipschitz function satisfying
\begin{equation}\label{ueq}
\dot u_\zeta(t)=-P_x(t,x_\zeta(t)),\end{equation}
for almost all $t\geq 0.$

For the future use, we make several notations following by some easy computations.
Consider a given point $(0,\eta)$ and the characteristic $x=x_\eta(t)$ initiated from it. Then we write $x=x_\xi(t)$ as the characteristic starting at some point $(0,\xi)$ with $\xi\neq\eta$, and set
$$f(t)=x_\xi(t)-x_\eta(t),\quad \text{ and }\quad g(t)=u_\xi(t)-u_\eta(t),$$
with $u_\xi(t)=u(t,x_\xi(t)), u_\eta(t)=u(t,x_\eta(t))$.
In view of \eqref{char1} and \eqref{ueq}, one has some useful identities
\begin{equation}\label{fgdot}
\dot f(t)=g(t) \quad {\rm and}\quad\dot g(t)=-P_x(t,x_\xi(t))+P_x(t,x_\eta(t)).
\end{equation}
holding for any time.

On the other hand, define
\begin{equation}\label{Txieta}T_{\xi,\eta}:=\inf\{t>0:x_\xi(t)=x_\eta(t)\}.
\end{equation}
Then for $t\in [0,T_{\xi,\eta})$,
we define
$$\omega(t)=\frac{g(t)}{f(t)}.$$
By the first equation in \eqref{fgdot}, it is clear that, for $t\in [0,T_{\xi,\eta})$,
$\dot f(t)=g(t)=\omega(t)f(t).$
This implies that,
\bel{derx}f(t)=(\xi-\eta)\exp\int_0^t \omega(s)ds,\q 0\leq t<T_{\xi,\eta}.\eel
The map $(t,\eta)\rightarrow (t,x_\eta(t))$ is locally Lipschitz on $[0,T_{\xi,\eta})\times(-\infty,\infty)$.
Moreover, according to the definition of $\omega(t)$ and \eqref{fgdot}, we can establish the differential equation for $\omega$.  For $t\in [0,T_{\xi,\eta})$, it follows that
\begin{equation}\label{west}
\ba{ll}\dot \omega(t)\ad=\frac{\dot g(t)}{f(t)}-\frac{g(t)\dot f(t)}{f^2(t)}=\frac{\dot g(t)}{f(t)}-\omega^2(t)\\
\ns\ad=-\frac{P_x(t,x_\xi(t))-P_x(t,x_\eta(t))}{x_\xi(t)-x_\eta(t)}-\omega^2(t).\ea
\end{equation}

For the initial value problem  \eqref{CH-eq}--\eqref{initialu}, from Section \ref{sec_before}, we know that the blowup is caused by the concentration of characteristics starting between the interval $[-q_0,q_0]$. Note that
$x_{q_0}(t)=q(t)$  and $x_{-q_0}(t)=-q(t)$ are two  characteristics starting at $q_0$ and $-q_0$ respectively and merging to $x=0$ as $t$ approaches the breaking time $T_0$, defined at \eqref{initialvalue}.
In what follows, we use two lemmas to show that:
\begin{itemize}
\item[{\bf 1.}] Characteristics starting at the point $(0,x)$ with $x\in[-q_0,q_0]$ will all 
merge to a characteristic $x=0$ as $t\rightarrow T_0^-$, and keep on this characteristic for $t\in[T_0, T_0+\delta]$ with some positive constant $\delta_0$.
\item[{\bf 2.}] There exists a constant $\delta_0$, such that, all characteristics starting at $(0,x)$ with $x\not\in[-q_0,q_0]$ will not cause blowup when $t\in[0,T_0+\delta_0]$, i.e. $T_{\xi,\eta}\geq T_0+\delta_0$, for any $\xi,\eta\not\in[-q_0,q_0]$.
\end{itemize}

\begin{lemma}\label{q0lem} For any dissipative solution of \eqref{CH-eq}--\eqref{initialu}, if $\xi$, $\eta\in [-q_0,q_0]$, then $T_{\xi,\eta}=T_0$, where $T_{\xi,\eta}$ and $T_0$ are defined in \eqref{Txieta} and \eqref{initialvalue}, respectively. Furthermore,
	\bel{derx2}\frac{dx_\xi(t)}{d\xi}=0,\text{ for any $\xi\in[-q_0,q_0]$ and $t\geq T_0$}.\eel
\end{lemma}

\begin{proof} By the definition of dissipative solution, we know that
$w(t,x)$ is bounded above in any bounded time interval, so $\displaystyle\omega(t)=\frac{g}{f}$ is locally bounded above. This means that in any time interval, there exists a constant $K$, such that
$$
\frac{g}{f}\leq K.$$
Then by the first equation of \eqref{fgdot}, in this time interval,
\bel{fkf}
\dot{f}\leq Kf.
\eel
Then by the ODE comparison theorem, we know that if $f=x_\xi(t)-x_\eta(t)$ vanishes at some time $T$,  it also vanishes at any time after $T$. This shows that if any two characteristics meet, they will merge into a single characteristic curve.

In particular, in our problem, we can use this result to prove that characteristics emanating from any point in $[-q_0,q_0]$ will merge to a characteristic $x=0$ after $t=T_0$. More precisely,
if $\xi,\eta\in [-q_0,q_0]$, by \eqref{charq0}, we see that
$$\lim_{t\rightarrow T_0^{-}}x_\xi(t)=\lim_{t\rightarrow T_0^-}x_\eta(t)=0$$
and thus by our discussion,
$$x_\xi(t)-x_\eta(t)=0\text{ for } t\geq T_0,$$
and
$$\frac{dx_\xi(t)}{d\xi}=0\q\text{for } \xi\in[-q_0,q_0] \text{ and }t\geq T_0.$$
So
$$T_{\xi,\eta}=T_0\q\text{for any $\xi,\eta\in [-q_0,q_0]$}.$$
This completes the proof of this lemma.
\end{proof}

The next lemma addresses the non-concentration of characteristics staring from $(0,\xi)$ with $\xi\in(-\infty, -q_0)\cup(q_0,\infty)$
before $T_0+\delta_0$ for some constant $\delta_0$ to be clarified in this lemma.

\begin{lemma}\label{Tuniformbound} We consider any dissipative solution of \eqref{CH-eq}--\eqref{initialu}.
There exists a uniform positive constant $\delta_0$ only depending on the initial energy $E_0$, such that, for all $\xi,\eta\in(-\infty, -q_0)\cup(q_0,\infty)$,
	 $$T_{\xi,\eta}\geq T_0+\delta_0.$$
\end{lemma}
\begin{proof}  First, by \eqref{uxt0},  for all $\xi,\eta\in(-\infty, -q_0)\cup(q_0,\infty)$, it holds that
	 $$T_{\xi,\eta}\geq T_0.$$
Thus we only have to consider solutions after $T_0$.
For two different characteristics $x=x_\xi(t)$ and $x=x_\eta(t)$ with  $\xi>\eta$, we turn to estimate the differential equation \eqref{west}. Recall that the energy is bounded by the initial energy $E_0$. Then we have
\begin{equation}\label{west1}\ba{ll}\ad -P_x(t,x_\xi(t))+P_x(t,x_\eta(t))-\frac12\int_{x_\eta(t)}^{x_\xi(t)}w^2(t,y)dy\\
\ns\ad\q=-\frac12\int_{x_\xi(t)}^{+\infty} e^{x_\xi(t)-y}\(u^2(t,y)+\frac12 w^2(t,y)\)dy+\frac12\int^{x_\xi(t)}_{-\infty} e^{y-x_\xi(t)}\(u^2(t,y)+\frac12 w^2(t,y)\)dy\\
\ns\ad\qq+\frac12\int_{x_\eta(t)}^{+\infty } e^{x_\eta(t)-y}\(u^2(t,y)+\frac12 w^2(t,y)\)dy-\frac12\int^{x_\eta(t)}_{-\infty} e^{y-x_\eta(t)}\(u^2(t,y)+\frac12 w^2(t,y)\)dy\\
\ns\ad\qq-\frac12\int_{x_\eta(t)}^{x_\xi(t)}w^2(t,y)dy\geq -E_0(x_\xi(t)-x_\eta(t))\q \text{ for a.e. }t\in[0, T_{\xi,\eta}).\ea \end{equation}
Moreover, by virtue of Schwartz's inequality, it holds that
\begin{equation}\label{west2}\Big(u(t,x_\xi(t))-u(t,x_\eta(t))\Big)^2\leq  (x_\xi(t)-x_\eta(t))\int_{x_\eta(t)}^{x_\xi(t)}w^2(t,y)dy. \end{equation}
Utilizing \eqref{west} and the estimates \eqref{west1}, \eqref{west2} to get
\begin{equation}\label{lower}
\begin{split}\dot \omega(t)&\geq -E_0+\frac12\frac{1}{x_\xi(t)-x_\eta(t)}\int_{x_\eta(t)}^{x_\xi(t)}w^2(t,y)dy-\omega^2(t)\q \\
&\geq -\frac12 \omega^2(t)-E_0,\q \text{ for a.e. }  t\in [0.T_{\xi,\eta}).
\end{split} \end{equation}
By \eqref{ut0}, $u(T_0,\cdot)=0$. This implies that $\omega(T_0)=0$. Thus we conclude  that
\bel{omegaxieta}\omega(t)\geq \sqrt{2E_0}\tan(-\sqrt {E_0/2} (t-T_0))>-\infty,\q\text{for } T_0\leq t < T_0+\frac{\pi}{\sqrt{2E_0}}.\eel
Note that for any $\xi>\eta\geq \eta'$, it follows that $T_{\xi,\eta}\leq T_{\xi,\eta'}$ from \eqref{derx},  \eqref{omegaxieta} holds for any $\xi,\eta\in(-\infty,-q_0)\cup(q_0,\infty)$.
This completes the proof of this lemma.
\end{proof}

\begin{remark}
We remark that the results in Lemmas \ref{q0lem} and \ref{Tuniformbound} are not new. For example, some similar results on the explosion time of Camassa-Holm equation for general initial data have been done by Grunert in \cite{GHR1-1,GHR1-0}. Here we present the proof for our special peakon and anti-peakon  initial data for reader's convenience.
\end{remark}

Finally, for future use, we give a technical lemma. The proof of this Lemma is elementary. We add it for reader's convenience.

\begin{lemma}\label{estriccati}
For any given $K_0>0$, we consider any continuous function  $g(t)$, $t\geq 0$, with $\displaystyle\sup_t|g(t)|=K_g<K_0$. Then there exists a positive constant $\delta_0$ depending only on $K_0$, such that any solution $w(t)$ to the following ODE
$$ \dot w(t)+\frac12w^2(t)=g(t),\q w(T_0)=0$$ satisfies
$$|w(t)|\leq 3K_g\d_0,\q\hbox{when }\ t\in[T_0,T_0+\d_0].$$
\end{lemma}
\begin{proof}
A direct computation on the ODE yields that
$$\frac{2\dot w(t)}{w^2(t)+2K_g}\geq -1,$$
thus, there exists a $\d_0$ depending only on $K_0$ such that when $t\in[T_0,T_0+\d_0]$,
$$w(t)\geq \sqrt{2K_g}\tan(-\sqrt{K_g/2} (t-T_0))\geq -3K_g\d_0.$$
On the other hand, it is easy to see that $\dot w(t)\leq K_g$
which implies $$w(t)\leq K_g \d_0\q\text{ for 	}\q t\in[T_0,T_0+\d_0].$$
This completes the proof of this lemma.
\end{proof}

When 
applying Lemma \ref{estriccati} to our problem \eqref{CH-eq}-\eqref{initialu}, the source term $g=-P_x$ is continuous and uniformly bounded. 

We might get two different $\delta_0$'s in  Lemmas  \ref{q0lem} and \ref{Tuniformbound}. For simplicity, we just choose $\delta_0$ to be the smaller one, such that both two lemmas hold.

\section{Proof of Theorem \ref{mainthm}}

 Now we are ready to prove Theorem \ref{mainthm}, on the uniqueness of dissipative solution for the Camassa-Holm equation with peakon-antipeakon initial data \eqref{initialu}.
 
 This section includes all new estimates in this paper. To prove the uniqueness result, our basic idea comes from \cite{Daf}. We have already discussed the difficulty of using the method of \cite{Daf} in Section \ref{section2}. Here we will give an application of the idea in \cite{Daf} on the peakon-antipeakon problem, where we use that the solution $u$ has small amplitude near the formation of singularity.

\begin{proof}
The first and also the main goal is to prove that the dissipative solution $u$ is uniquely equal zero, in a time interval $(T_0,T_0+\delta_0]$, with a sufficiently small positive constant $\delta_0$ depending on $E_0$ only.

According to Lemma \ref{q0lem}, any characteristics initially from the interval $[-q_0,q_0]$ will merge into a single characteristic curve $x=0$ for all $t\geq T_0$.

Thus, we only have to consider the characteristics, such as  $x_\xi$, $x_\eta$, emanating from $(-\infty, -q_0)\cup (q_0,\infty)$.  By Lemma \ref{Tuniformbound}, we know there exists $\delta_0$, such that any two different characteristics $x_\xi$ and $x_\eta$ will not meet each other in $(T_0,T_0+\delta_0]$. Later, the value of $\delta_0$ may need to be chosen smaller as needed. For simplicity, we will keep using $\delta_0$ to denote it.

To begin with, we take a uniform small constant $\d_0>0$ such that $$\tan(\sqrt{E_0/2}\delta_0)\leq 3\sqrt{E_0/2}\delta_0,$$
where recall that $E(t)$ defined in \eqref{edef} is the energy, and $E_0=E(0)$ is the initial energy.

\paragraph{\bf 1.}
By Lemma \ref{Tuniformbound}, it follows that the map $(t,\zeta)\mapsto (t,x_\zeta(t)) $ restricted to the set $\{(t,\xi);~\xi\in(-\infty,-q_0)\cup(q_0,\infty),t\in[0,T_0+\d_0]\}$ is locally Lipschitz one-to-one  and the inverse is locally Lipschitz.
	
Let $\Gamma$ be the set of $(t,x)$ such that $w(t,x)=u_x(t,x)$ exists and \begin{equation*}
\lim_{\epsilon\rightarrow0}\int_x^{x+\epsilon}w^2(t,y)\,dy
=w^2(t,x),\end{equation*}
 has a full measure in  $[0,T_0+\d_0]\times(-\infty,\infty)$.
Thus there exists a subset $J$ of $(-\infty,-q_0)\cup(q_0,\infty)$ such that $\zeta\in J$ implies $(t,x_\zeta(t))\in\Gamma$. Here $J$ has full measure in $(-\infty,-q_0)\cup(q_0,\infty)$, for almost all $t$ in $[0,T_0+\d_0$].

Given $\eta\in J$ and let $\xi\rightarrow\eta$, it follows that
	\bel{limit}\omega(t)\rightarrow w_\eta(t),\eel
 boundly almost everywhere for $t\in[0,T_0+\d_0]$, and \begin{equation}\label{wpoint}
 \frac{1}{x_\xi(t)-x_\eta(t)}\int_{x_\eta(t)}^{x_\xi(t)}w^2(t,y)\,dy\rightarrow w_\eta^2(t)
 \end{equation}
almost everywhere on $[0,T_0+\d_0],$	where $w_\eta(t)=w(t,x_\eta(t)).$

\paragraph{\bf 2.}	
Integrating \eqref{west} over any interval $[t_0,t_1]\subset [0,T_0+\d_0]$, with both $(t_0,x_\eta(t_0))$ and $(t_1,x_\eta(t_1))$ in $\Gamma$. By the fact that $P_{xx}=P-(u^2+\frac{1}{2}u_x^2)$, we have
\bel{equalintegration}\ba{ll}\omega(t_1)-\omega(t_0)\ad
=\int_{t_0}^{t_1}\Big(\frac{1}{x_\xi(t)-x_\eta(t)}\int_{x_\eta(t)}^{x_\xi(t)}\big(u^2+\frac12w^2-P\big)(t,y)\,dy\Big)\,dt
-\int_{t_0}^{t_1}\omega^2(t)\,dt.\ea\eel
By using \eqref{west2}, we can derive	$$\int_{t_0}^{t_1}\omega^2(t)\,dt\leq\int_{t_0}^{t_1}\Big(\frac{1}{x_\xi(t)-x_\eta(t)}
\int_{x_\eta(t)}^{x_\xi(t)}w^2(t,y)\,dy\Big)\,dt,$$
which together with \eqref{equalintegration} gives that		\bel{lowerbound}\omega(t_1)-\omega(t_0)\geq\int_{t_0}^{t_1}\Big(\frac{1}{x_\xi(t)-x_\eta(t)}\int_{x_\eta(t)}^{x_\xi(t)}\big(u^2-\frac12w^2-P\big)(t,y)\,dy\Big)\,dt.\eel	By the continuity of  $u$ and $P$, one has $$\lim_{\xi\rightarrow\eta}\frac{1}{x_\xi(t)-x_\eta(t)}\int_{x_\eta(t)}^{x_\xi(t)}\big(u^2-P\big)(t,y)\,dy=u^2(t,x_\eta(t))-P(t,x_\eta(t)).$$
Since  $P$ is bounded uniformly and $u$ is  bounded on any compact set, by the dominant convergent theorem, we have	\begin{equation}\label{updom}
\lim_{\xi\rightarrow\eta}\int_{t_0}^{t_1}\Big(\frac{1}{x_\xi(t)-x_\eta(t)}\int_{x_\eta(t)}^{x_\xi(t)}\big(u^2-P\big)(t,y)\,dy\Big)\,dt
=\int_{t_0}^{t_1}\Big(u^2(t,x_\eta(t))-P(t,x_\eta(t))\Big)\,dt.
\end{equation}
Letting $\xi\rightarrow \eta$ in \eqref{lowerbound} and by using \eqref{limit}, \eqref{wpoint} and \eqref{updom}, for almost all $t_0,t_1\in [0,T_0+\d_0]$, we derive
\bel{upper}w_\eta(t_1)-w_\eta(t_0)\geq\int_{t_0}^{t_1}
\Big(u^2(t,x_\eta(t))-P(t,x_\eta(t))\Big)\,dt-\frac12\int_{t_0}^{t_1}w_\eta^2(t)\,dt ,\quad\text{ for  }\eta\in J.\eel	

On the other hand, letting $\xi\rightarrow \eta$ in \eqref{equalintegration}, and using \eqref{limit}, \eqref{wpoint}, \eqref{updom}, the dominant convergence theorem and Fatou's Lemma,  for almost all $t_0,t_1\in[0,T_0+\d_0]$, it holds that	\bel{lower}\ba{ll}w_\eta(t_1)-w_\eta(t_0)\ad\leq\int_{t_0}^{t_1}\Big(u^2(t,x_\eta(t))-P(t,x_\eta(t))\Big)\,dt-\frac12 \int_{t_0}^{t_1}w_\eta^2(t)dt,\text{ for  }\eta\in J.\ea\eel
From \eqref{upper} and \eqref{lower}, we have  for almost all $t_0,t_1\in[0,T_0+\d_0]$,
\bel{riccatiw}	\ba{ll}w_\eta(t_1)-w_\eta(t_0)\ad=\int_{t_0}^{t_1}\Big(u^2(t,x_\eta(t))-P(t,x_\eta(t))\Big)\,dt-\frac12 \int_{t_0}^{t_1} w_\eta^2(t)\,dt,\text{ for  }\eta\in J.\ea \eel
This provides the equation for $w(t,x_\eta(t))$. To establish the bound for $w(t,x_\eta(t))$, in view of Lemma \ref{estriccati}, it suffices to give the estimates for $|u^2(t,x_\eta(t))-P(t,x_\eta(t))|$. we proceed as follows.

\paragraph{\bf 3.}
Note that the energy is bounded by the initial energy $E_0$. Recalling the definition of $P$ in \eqref{CH-eq2}, we obtain the uniform bounds
 \begin{equation}\label{ppxbound}\|P(t)\|_{L^\infty},\| P_x(t)\|_{L^\infty}\leq \|\frac{1}{2}e^{-|x|}\|_{L^\infty}\cdot\|u^2(t)+\frac{u_x^2(t)}{2}\|_{L^1}
 \leq \frac{1}{2}E_0. \end{equation}
Since by \eqref{ut0}, $ u(T_0,\cdot)=0$. The equation \eqref{ueq} implies
\begin{equation}\label{1}
|u(T_0+\d_0,x_\xi(T_0+\d_0))|=
\left|u(T_0,x_\xi(T_0))-\int_{T_0}^{T_0+\d_0}P_x(s,x_\xi(s))\,ds\right|\leq \frac{1}{2} E_0\d_0.\end{equation}
A further calculation from \eqref{ppxbound} and \eqref{1} yields
\begin{equation}\label{u-p1}
|u^2(t,x_\eta(t))-P(t,x_\eta(t))|\leq \frac{1}{4} E_0^2\d_0^2+\frac{1}{2} E_0\leq E_0.
\end{equation}
This together with  \eqref{riccatiw} and Lemma \ref{estriccati} gives rise to
\begin{equation}\label{w1}|w_\eta(t)|\leq 3E_0\d_0,\q \text{ for } t\in[T_0, T_0+\d_0].\end{equation}
	Let
	\begin{equation*}
	D_1:=3E_0.
	\end{equation*}
Thus we already proved that
\begin{equation}\label{uw1}
|u_\eta(t)|,|w_\eta(t)|\leq D_1\d_0,\q \text{ for } t\in[T_0, T_0+\d_0].
\end{equation}
The bound on \eqref{uw1} actually follows from the estimates for $P$ and $P_x$ of \eqref{ppxbound}. In the following, we re-derive this bound in terms of \eqref{derx}, \eqref{derx2} and \eqref{uw1}.

\paragraph{\bf 4.} From \eqref{derx}, for $\xi,\eta\in J$, sending $\xi\rightarrow\eta$  and by \eqref{limit}, we deduce $$\frac{dx_\eta(t)}{d\eta}=\exp\Big(\int_0^tw_\eta(s)ds\Big),\q t\in[0,T_0+\d_0].$$
	Together with \eqref{derx2}, it follows that for $t\in[T_0, T_0+\d_0]$, $$\frac{dx_{\eta}(t)}{d{\eta}}=\left\{\ba{ll} 0,\ad \q\text{for }\eta\in[-q_0,q_0]\\\ns\ds
	\exp\(\int_0^tw_\eta(s)ds\)\geq0, \ad\q \text{for }\eta\in(-\infty,-q_0)\cup(q_0,\infty), \ea\right.$$
	Then by the definition of $P$ in \eqref{CH-eq2} and \eqref{uw1}, it is straightforward to verify that
	\bel{estP}\ba{ll}P(t,x_\eta(t))\ad=\int_{-\infty}^\infty e^{-|x_\eta(t)-y|}\(u^2(t,y)+\frac{w(t,y)^2}2\)dy\\
	\ns\ad\leq\(\int_{-\infty}^{-q_0}+\int^{\infty}_{q_0}\) e^{-|x_\eta(t)-x_{\zeta}(t)|}\(u^2(t,x_{\zeta}(t))+\frac{w^2(t,x_{\zeta}(t))}2\)\frac{dx_{\zeta}(t)}{d\zeta}{d\zeta}\\
	\ns\ad\leq \frac32D_1^2\d_0^2\int_{-\infty}^\infty e^{-|x_\eta(t)-y|}dy\\
	\ns\ad \leq 3D_1^2\d_0^2. \ea\eel
	Similarly we can conclude that for $\eta\in J$,
	\bel{estPx}\ba{ll}|P_x(t,x_\eta(t))|\ad\leq \int_{-\infty}^{x_\eta(t)} e^{-(x_\eta(t)-y)}\(u^2(t,y)+\frac{w(t,y)^2}2\)dy+\int^{\infty}_{x_\eta(t)} e^{x_\eta(t)-y}\(u^2(t,y)+\frac{w(t,y)^2}2\)dy\\
	\ns\ad\leq3 D_1^2\d_0^2.\ea\eel
Now we have worked out the other detailed estimates for $P$ and $P_x$, which are different from \eqref{ppxbound}. A new estimate on $u_\eta(t)$ and $w_\eta(t)$ will thus follow from the new bounds on $P$ and $P_x$. Toward this goal, we proceed as in \eqref{1}--\eqref{uw1}.
	
For $t\in[T_0,T_0+\d_0]$, similar to the estimate in \eqref{1}, by \eqref{estPx}, we obtain \bel{u2}\ba{ll}|u(t,x_\eta(t))|\ad=\left|u(T_0,x_\eta(T_0))-\int_{T_0}^{t}P_x(s,x_\eta(s))ds\right|\leq 3D_1^2\d_0^3.\ea\eel
We can thus repeat the estimate in \eqref{u-p1}--\eqref{uw1}.  Apply \eqref{estP},  \eqref{u2} in \eqref{riccatiw} again and use Lemma \ref{estriccati}, it holds that
	$$|w_\eta(t)|\leq 3(3D_1^2\d_0^2+3D^2_1\d_0^3)\d_0\leq 18D_1^2\d_0^3, \q\text{ for } t\in[T_0, T_0+\d_0].$$
Let
$$D_2:=18D_1^2\d_0^2.$$
Thus we already proved that
\begin{equation}\label{uw2}
|u_\eta(t)|,|w_\eta(t)|\leq D_2\d_0,\q \text{ for } t\in[T_0, T_0+\d_0].
\end{equation}
\paragraph{\bf 5.} Recursively repeating the whole process for \eqref{estP}--\eqref{uw2}. We can get a sequence of $D_i$ following
	$$D_{i+1}=18 D_i^2\d_0^2$$
	and
	$$|u_\eta(t)|,|w_\eta(t)|\leq D_i\d_0,\q \text{ for } t\in[T_0, T_0+\d_0] \text{ and  any } i.$$
Since $\d_0$ is a uniform constant depending on $E_0$ only, we can choose $\d_0$ small enough such that $D_i$ converges to 0 as $i\rightarrow\infty$. Thus  for $\eta\in(-\infty,-q_0)\cup(q_0,\infty)$,
	$$|u(t,x_\eta(t))|=|w(t,x_\eta(t))|=0,\q t\in[T_0,T_0+\d_0].$$
	
		\bigskip
	
Finally,	since $u(t,x)$ is continuous by the definition of dissipative solution, we conclude that the dissipative solution uniquely equals to $0$ for  $t\in[T_0,T_0+\delta_0]$. Hence, we can claim that the solution uniquely vanishes for any time after $t=T_0$. In fact, the proof of this claim is classical. Alternatively, to make this paper self-contained, one can simply repeat our proof for the uniqueness when $t\in[T_0,T_0+\d_0]$ to show that the solution $u$ must vanish on  $[T_0+\d_0,T_0+2\d_0]$ , $\cdots$, $[T_0+n\d_0,T_0+(n+1)\d_0]$ and so on. This  completes the proof of Theorem \ref{mainthm}.

\end{proof}

\paragraph{\bf Acknowledgment:} The authors wish to thank for some helpful discussions with Tao Huang, Weizhang Huang and Qingtian Zhang, in preparing this paper.

\end{document}